\documentclass[a4paper,12pt]{amsart}

\usepackage{fancyhdr}
\usepackage{a4wide}
\usepackage{tikz}
\usepackage{amsmath,amsthm}
\usepackage{float}
\usepackage{color}
\usepackage[noadjust]{cite}
\usepackage{enumitem}
\usepackage{cleveref}
\newcommand{\mylabel}[2]{#2\def\@currentlabel{#2}\label{#1}}

\DeclareMathOperator{\I}{I}

\newcommand{\A}{\alpha}

\DeclareMathOperator{\avi}{avi}

\DeclareMathOperator{\Ti}{T}

\DeclareMathOperator{\ii}{i}

\newcommand{\changefont}{%
    \fontsize{9}{12}\selectfont}

\newtheorem{thm}{Theorem}[section]
\newtheorem{prop}[thm]{Proposition}
\newtheorem{lem}[thm]{Lemma}
\newtheorem{cor}[thm]{Corollary}
\theoremstyle{remark}

\newtheorem{exa}[thm]{Example}

\newtheorem*{thm*}{Theorem}

\newif\ifdetails
\detailstrue
\newcommand{\DETAIL}[1]%
{\ifdetails\par\fbox{\begin{minipage}{0.9\linewidth}\textit{Detail:}
      #1\end{minipage}}\par\fi}
\newcommand{\TODO}[1]%
{\ifdetails\par\fbox{\begin{minipage}{0.9\linewidth}\textbf{TODO:}
      #1\end{minipage}}\par\fi}

\newcommand{\old}[1]{{}}

\title{The average size of independent sets of graphs}

\author{Eric O. D. Andriantiana}
\address{Eric O. D. Andriantiana\\
Department of Mathematics (Pure and Applied)\\
Rhodes University, PO Box 94\\
6140 Grahamstown\\
South Africa}
\email{E.Andriantiana@ru.ac.za}

\author{Valisoa Razanajatovo Misanantenaina}
\address{Valisoa Razanajatovo Misanantenaina\\
Department of Mathematical Sciences\\
Stellenbosch University\\
Private Bag X1\\
Matieland 7602\\
South Africa}
\email{valisoa@sun.ac.za}

\author{Stephan Wagner}
\address{Stephan Wagner\\
Department of Mathematical Sciences\\
Stellenbosch University\\
Private Bag X1\\
Matieland 7602\\
South Africa}
\email{swagner@sun.ac.za}

\thanks{This work was supported by the National Research Foundation of South Africa (grants 96236 and 96310).}


\begin{document}

\begin{abstract}
In this paper, we study the average size of independent (vertex) sets of a graph. This invariant can be regarded as the logarithmic derivative of the independence polynomial evaluated at $1$. We are specifically concerned with extremal questions. The maximum and minimum for general graphs are attained by the empty and complete graph respectively, while for trees we prove that the path minimises the average size of independent sets and the star maximises it. While removing a vertex does not always decrease the average size of independent sets, we prove that there always exists a vertex for which this is the case.
\end{abstract}

\maketitle

\section{Introduction}

The number of independent sets is a graph invariant that has been studied extensively. It has been dubbed the \emph{Fibonacci number} of a graph \cite{prodinger1982fibonacci} due to the fact that the number of independent sets in a path is always a Fibonacci number, and is now known as \emph{Merrifield-Simmons index} in mathematical chemistry in honour of the work of chemists Merrifield and Simmons \cite{merrifield1989topological}. Moreover, its connection to the \emph{hard core model} in statistical physics is well established (see~\cite{brightwell2002hard} for a general reference).

Extremal problems (concerned with finding maximum or minimum values) regarding the number of independent sets have been of particular interest, especially in the aforementioned context of mathematical chemistry. Graphs with various restrictions have been studied, as well as graph classes such as trees, unicyclic or bicyclic graphs; see \cite{wagner2017upper} for a recent survey.

In this paper, we study similar questions for the average size of independent sets. This invariant, while interesting in its own right, comes up in various contexts: in \cite{davies2018average}, an asymptotic lower bound is given for triangle-free graphs, which can be used to obtain bounds on Ramsey numbers. In \cite{davies2017independent}, the same authors established an upper bound as a tool to prove that the disjoint union of complete bipartite graphs $K_{d,d}$ maximises the number of independent sets of a $d$-regular graph.

An invariant of a similar nature is the average order of a subtree, as introduced to the literature by Jamison \cite{jamison1983average,jamison1984monotonicity}.  In particular, he proved that the average order of subtrees of an $n$-vertex tree is at least $(n+2)/3$, with equality for the path, which parallels our Theorem~\ref{thm:path}.
There has been a fair amount of recent activity around this invariant \cite{vince2010average,haslegrave2014extremal,wagner2016local,Mol} and its generalisations \cite{stephens2018mean}.

This paper is structured as follows: in the following section, we collect some basic results that will be needed for our analysis. In Section \ref{general}, we consider the behaviour under removal of vertices or edges. It turns out that the average size of independent sets is not monotone under these operations, as we will show by explicit counterexamples. This is in contrast to the number of independent sets, but also to the aforementioned average subtree order~\cite{jamison1984monotonicity}. However, we prove that it is always possible to find a vertex whose removal decreases the average size of independent sets. We also show that---not unexpectedly---the empty and complete graph attain the maximum and minimum respectively among graphs of a given order. 
Finally, we focus on trees in Section \ref{trees}, where it is shown that the path and the star are extremal. While the proof for the star is fairly short and straightforward, the situation for the path is much more complex.
The paper concludes with a brief discussion of a generalised invariant.

\section{Preliminaries}

Let $G$ be a graph, and let $\ii(G,k)$ be the number of independent sets of size $k$.
The independence polynomial of $G$ is defined by $$\I(G,x)=\sum_{k}\ii(G,k)x^k,$$ 
see \cite{levit2005independence} for a survey on the independence polynomial and its properties. The total number of independent subsets 
of $G$ is simply the value of the independence polynomial at $1$:
$$
\I(G,1)=\sum_{k}\ii(G,k),
$$
and the first derivative, evaluated at $x=1$, is
$$
\I'(G,1)=\sum_{k}k\ii(G,k).
$$
Hence the average size of independent vertex subsets in $G$ is the logarithmic derivative
$$
\avi(G)=\frac{\I'(G,1)}{\I(G,1)}.
$$

For ease of notation, we will write $\I(G)$ instead of $\I(G,1)$, as well as $\Ti(G)$ instead of $\I'(G,1)$ (total size of all independent sets).

\begin{exa}
Let us compute the average size of independent sets of the $n$-vertex edgeless graph $E_n$ and star $S_n$. We have $$\I(E_n,x)=(1+x)^n, \quad \I(S_n,x)=(1+x)^{n-1}+x,$$ which give
$$\I(E_n)=2^{n} \text{ and } \I(S_n)=2^{n-1}+1,$$ 
$$\Ti(E_n)=n2^{n-1} \text{ and } \Ti(S_n)=(n-1)2^{n-2}+1,$$
and hence
$$
\avi(E_n)=\frac{n}{2}, \quad \avi(S_n)=\frac{n-1}{2}+\frac{3-n}{2^{n}+2}.
$$
\end{exa}

The following standard recursion, which is obtained by distinguishing independent sets containing a vertex $v$ and those not containing it, is very useful in calculating the independence polynomial of graphs.

\begin{prop}\label{cha1.pro1}
Let $v$ be a vertex of $G$ and $N[v]=\{u| uv \in E(G)\}\cup\{v\}$ its closed neighbourhood. We have
\[\I(G,x)=\I(G-v,x)+x\I(G-N[v],x).\]
\end{prop}

As an immediate consequence, we obtain recursions for the aforementioned invariants $I(G)$ and $T(G)$. 
\begin{prop}\label{sizeRec}
Let $v$ be a vertex of $G$ and $N[v]$ its closed neighbourhood. We have
\begin{align}
\I(G)&=\I(G-v)+\I(G-N[v]),\label{chap2.eq}\\
\Ti(G)&=\Ti(G-v)+\Ti(G-N[v])+ \I(G-N[v]).\label{chap2.eq2}
\end{align}
\end{prop}

\begin{proof}
The first equation \eqref{chap2.eq} is obtained from Proposition \ref{cha1.pro1} by putting $x=1$, the second by differentiating first and plugging in $x=1$ afterwards.
\end{proof}

Thus, we get the following identities for the average size of independent sets:
\begin{prop}\label{chap2.Eq:muRec}
Let $v$ be a vertex of $G$ and $N[v]$ its closed neighbourhood. We have
\begin{align*}
\avi(G)&=\frac{\Ti(G-v)+\Ti(G-N[v])+ \I(G-N[v])}{\I(G-v)+\I(G-N[v])}\\
&=\frac{\avi(G-v)\I(G-v)+(\avi(G-N[v])+1)\I(G-N[v])}{\I(G-v)+\I(G-N[v])}.
\end{align*}
\end{prop}

We conclude this section with the following lemma, which will be useful later.

\begin{lem}
If $G_1,G_2,\dots,G_k$ are the disjoint components of a graph $G$, then
\[\avi(G)=\sum_k\avi(G_k).\]
\end{lem}

\begin{proof}
It is well known that $$\I(G,x)=\prod_k\I(G_k,x),$$
see \cite{levit2005independence}. Now take the logarithm on both sides, differentiate and plug in $x=1$ to obtain the 
desired result.
\end{proof}

\section{Vertex or edge removal}\label{general}

Many graph invariants satisfy a monotonicity property with respect to vertex or edge removal. This means that removing a vertex (an edge) either always decreases of always increases the value of the invariant. The total number of independent sets is a simple example for such monotonicity properties: clearly, we have
$$I(G - v) < I(G) \text{ and } I(G - e) > I(G)$$
for every vertex $v$ and every edge $e$ of a graph $G$. As we will see in this section, the average size of independent sets is not a monotone invariant. However, we will show that there always exists a vertex in the graph whose removal decreases $\avi$. Then, we characterize extremal graphs among all $n$-vertex graphs.

Let us first show that the the average size of independent sets is not monotone with respect to vertex removal. If $v$ and $u$ are 
respectively a leaf and the centre of the $4$-vertex star $S_4$, then we have
$$\avi(S_4-v)=\avi(S_3)=1<\frac{13}{9}=\avi(S_4),$$
but
$$\avi(S_4-u)=\avi(E_3)=\frac{3}{2}>\frac{13}{9}=\avi(S_4).$$

The average size of independent sets is also not monotone with respect to edge removal: considering the tree in Figure \ref{treeT}, we have
$$\avi(T-e_1)=\frac{13}{9}+\frac{2}{3}=\frac{19}{9}<\frac{55}{26}=\avi(T),$$
but
$$\avi(T-e_2)=\frac{33}{17}+\frac{1}{2}=\frac{83}{34}>\frac{55}{26}=\avi(T).$$

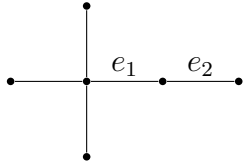
\begin{figure}[H]
\centering
\begin{tikzpicture}[scale=1]
\node[fill=black,circle,inner sep=1pt] (t1) at (0,0) {};
\node[fill=black,circle,inner sep=1pt] (t2) at (1,0) {};
\node[fill=black,circle,inner sep=1pt] (t3) at (2,0) {};
\node[fill=black,circle,inner sep=1pt] (t4) at (3,0) {};
\draw (t1)--(t2)--(t3)--(t4);
\node[fill=black,circle,inner sep=1pt] (t5) at (1,1) {};
\node[fill=black,circle,inner sep=1pt] (t6) at (1,-1) {};
\draw (t5)--(t2)--(t6);
\node at (1.5,0.2) {$e_1$};
\node at (2.5,0.2) {$e_2$};
\end{tikzpicture}
\caption{A tree $T$}\label{treeT}
\end{figure}

While the inequality $\avi(G-v)<\avi(G)$ does not always hold (as the example of the star shows), we can show that for every graph $G$, there exists a vertex $v$ with this property. To this end, we require the following theorem:
\begin{thm}
\label{chap2.Thm:Gen}
Let $X$ be a nonempty finite set, and $\mathcal{P}(X)$ its powerset. For a nonempty subset $\mathcal{A}$ of $\mathcal{P}(X)$, 
we define
$$
av(\mathcal{A})=\frac{1}{|\mathcal{A}|}\sum_{A\in \mathcal{A}}|A|.
$$
Let $\mathcal{B}$ be a subset of $\mathcal{P}(X)$, such that the cardinalities of the elements of $\mathcal{B}$ are not all the same.
Then there exists $x_0\in X$ such that $\mathcal{B}\cap \mathcal{P}(X-\{x_0\})$ is not empty and
$$
av(\mathcal{B})>av\big(\mathcal{B}\cap \mathcal{P}(X-\{x_0\})\big).
$$
\end{thm}
\begin{proof}
It is convenient to use the abbreviations
$$
n_k(\mathcal{A})=\big|\{A\in \mathcal{A}:|A|=k\}\big| 
\text{ and }S(\mathcal{A})=\sum_{A\in \mathcal{A}}|A|=\sum_{k\geq 0}k\cdot n_k(\mathcal{A}).
$$
We will prove that 
\begin{equation}\label{eq:averaged_ineq}
av(\mathcal{B})>\frac{\sum_{x\in X}S(\mathcal{B}\cap\mathcal{P}(X-\{x\}))}{\sum_{x\in X}|\mathcal{B}\cap\mathcal{P}(X-\{x\})|}.
\end{equation}
Note here that the denominator is not $0$: if $\mathcal{B}\cap\mathcal{P}(X-\{x\})$ was empty for all $x$, then $\mathcal{B}$ could only contain the set $X$ and nothing else, contradicting our assumption that the cardinalities of the elements of $\mathcal{B}$ are not all the same.
The claim of the theorem follows immediately, since there must be an $x_0 \in X$ such that
$$
\frac{\sum_{x\in X}S(\mathcal{B}\cap\mathcal{P}(X-\{x\}))}{\sum_{x\in X}|\mathcal{B}\cap\mathcal{P}(X-\{x\})|}
\geq \frac{S(\mathcal{B}\cap\mathcal{P}(X-\{x_0\}))}{|\mathcal{B}\cap\mathcal{P}(X-\{x_0\})|}.
$$
Now let us prove~\eqref{eq:averaged_ineq}. In the sum
$$
\sum_{x\in X}S(\mathcal{B}\cap\mathcal{P}(X-\{x\}))=
\sum_{x\in X}\sum_{k\geq 0}k\cdot n_k(\mathcal{B}\cap\mathcal{P}(X-\{x\})),
$$
the size of each $B\in\mathcal{B}$ contributes $|X|-|B|$ times. Hence 
\begin{align*}
\sum_{x\in X}S(\mathcal{B}\cap\mathcal{P}(X-\{x\}))
&=\sum_{k\geq 0}(|X|-k)k\cdot n_k(\mathcal{B})=|X|S(\mathcal{B})-\sum_{k\geq 0}k^2n_k(\mathcal{B})\\&
=|X|av(\mathcal{B})|\mathcal{B}|-\sum_{k\geq 0}k^2n_k(\mathcal{B}).
\end{align*}
Similarly,
\begin{align*}
\sum_{x\in X}|\mathcal{B}\cap\mathcal{P}(X-\{x\})|
&=\sum_{x\in X}\sum_{k\geq 0}n_k(\mathcal{B}\cap\mathcal{P}(X-\{x\}))\\
&=\sum_{x\in X}\sum_{k\geq 0}(|X|-k)n_k(\mathcal{B})\\
&=|X||\mathcal{B}|-S(\mathcal{B}).
\end{align*}

By the Cauchy-Schwarz inequality, we have
\begin{align*}
S(\mathcal{B})^2=\Big(\sum_{k\geq 0}k\cdot n_k(\mathcal{B})\Big)^2
\leq \Big( \sum_{k\geq 0}n_k(\mathcal{B}) \Big) \Big( \sum_{k\geq 0}k^2n_k(\mathcal{B}) \Big) =|\mathcal{B}|\sum_{k\geq 0}k^2n_k(\mathcal{B}), 
\end{align*}
and equality holds if and only if there is only one $k$ such that $n_k(\mathcal{B}) \neq 0$, which means all the elements of $\mathcal{B}$ have the same size.
Since this is ruled out by our assumptions, we actually have
$$
av(\mathcal{B})S(\mathcal{B})< \sum_{k\geq 0}k^2n_k(\mathcal{B}).
$$
Therefore we get
\begin{align*}
\frac{\sum_{x\in X}S(\mathcal{B}\cap\mathcal{P}(X-\{x\}))}{\sum_{x\in X}|\mathcal{B}\cap\mathcal{P}(X-\{x\})|}
&=\frac{|X|av(\mathcal{B})|\mathcal{B}|-\sum_{k\geq 0}k^2n_k(\mathcal{B})}{|X||\mathcal{B}|-S(\mathcal{B})}\\
&< \frac{|X|av(\mathcal{B})|\mathcal{B}|-av(\mathcal{B})S(\mathcal{B})}{|X||\mathcal{B}|-S(\mathcal{B})}=av(\mathcal{B}),
\end{align*}
which concludes the proof of~\eqref{eq:averaged_ineq} and thus the theorem.
\end{proof}
\begin{cor}
If $G$ is a nonempty graph, then there exists a vertex $v$ in $G$ such that 
$$
\avi(G-v)<\avi(G).
$$
\end{cor}
\begin{proof}
Apply Theorem \ref{chap2.Thm:Gen}, with $\mathcal{B}$ being the set of independent vertex subsets of $G$.
\end{proof}

We have seen that there is always a vertex in a graph whose removal decreases the average size of independent sets $\avi$. However, the dual statement for edge removal does not hold, namely there is not always an edge whose removal increases $\avi$ (nor is there always an edge whose removal decreases $\avi$). As counterexamples, we can consider the stars $S_6$ and $S_4$: for any edge $e$ in $S_6$ ($S_4$, respectively) we have

\begin{align*}
\avi(S_6)&=\frac{27}{11}>\frac{83}{34}=\avi(S_6-e),\\
\avi(S_4)&=\frac{13}{9}<\frac{3}{2}=\avi(S_4-e).
\end{align*}
So every edge removal in $S_6$ decreases $\avi$, while every edge removal in $S_4$ increases $\avi$.

Despite this, the edgeless graphs and the complete graphs are extremal:
\begin{thm}\label{chap2.thmemp}
For every $n$-vertex graph $G$ that is not the edgeless graph $E_n$ or the complete graph $K_n$, 
$\frac{n}{n+1}=\avi (K_n) < \avi(G)<\avi(E_n)=\frac{n}{2}$. 
\end{thm}

\begin{proof}
The first inequality is straightforward from the fact that the only independent sets of $K_n$ are $n$ independent sets of size 1 and the empty set: all other graphs with $n$ vertices have these independent sets and some larger ones. We prove the second inequality by induction. For $n=1$, there is no possible graph different from $E_n$, so there is nothing to prove. Now, assume that the inequality is true for all $n \leq k$, for some $k \geq 1$. Let $G$ be a 
$(k+1)$-vertex graph that is not edgeless. Let $v \in V(G)$ be a vertex such that 
$\deg(v) \geq 1$. We have
$$
\avi(G)=\frac{\Ti(G-v)+\Ti(G-N[v])+\I(G-N[v])}{\I(G-v)+\I(G-N[v])}.
$$
Using the induction hypothesis, we obtain
\begin{align*}
\avi(G)& \leq \frac{\frac{k}{2}\I(G-v)+(\frac{k-1}{2}+1)(\I(G-N[v]))}{\I(G-v)+\I(G-N[v])}\\
&= \frac{k}{2} + \frac{\frac{1}{2}\I(G-N[v])}{\I(G-v)+\I(G-N[v])}\\
&< \frac{k}{2} + \frac{1}{2}=\frac{k+1}{2}=\avi(E_{k+1}).
\end{align*}	 
\end{proof}

\section{Trees}\label{trees}

In this section, we are concerned with extremal trees regarding the average size of independent sets.
Let us first consider the problem of maximizing the average size of independent sets among all $n$-vertex trees.
\begin{thm}
For every $n$-vertex tree $T$, $\avi(S_n) \geq \avi(T)$. 
\end{thm}

\begin{proof}
In the cases where $n=1,2,3$, we must have $T=S_n$, and thus the claim holds. 

Assume the inequality holds for all $n \leq k$, for some $k \geq 3$. Now suppose that $T\neq S_n$ is a tree with $n=k+1$ vertices. Let $v \in V(T)$ be a leaf of $T$ and $u$ its neighbour. Then 
$T-v$ is still a tree and by Proposition \ref{sizeRec}
$$
\I(T-v)-\I(T-v-u)=\I(T-v-N[u])> 1.
$$ Moreover, the star minimises the number of independent sets among all $n$-vertex trees (see \cite{prodinger1982fibonacci}), i.e.~ $\I(T-v)\leq \I(S_{n-1}) = 2^{n-2} + 1$. Thus
\begin{equation}
\frac{\I(T-N[v])}{\I(T-v)}=1-\frac{\I(T-v)-\I(T-N[v])}{\I(T-v)}
<1-\frac{1}{\I(T-v)}\leq 1-\frac{1}{2^{n-2}+1}.\label{avisn}
\end{equation}
Using the recursion of Proposition \ref{chap2.Eq:muRec} and Theorem \ref{chap2.thmemp}, we obtain
\begin{align*}
\avi(T)&=\frac{\avi(T-v)\I(T-v)+(\avi(T-N[v])+1)\I(T-N[v])}{\I(T-v)+\I(T-N[v])}\\
&\leq \frac{\avi(S_{n-1})\I(T-v)/\I(T-N[v])+\frac{n-2}{2}+1}{\I(T-v)/\I(T-N[v])+1}.
\end{align*}

Since $\avi(S_{n-1}) <\frac{n-2}{2}+1$, $\frac{\avi(S_{n-1})x+\frac{n-2}{2}+1}{x+1}$ is decreasing as a function of $x$ for $x \geq 0$. Combined with the inequality in \eqref{avisn}, this shows that
\begin{align*}
\avi(T)< \frac{\avi(S_{n-1})(2^{n-2}+1)/2^{n-2}+\frac{n-2}{2}+1}{(2^{n-2}+1)/2^{n-2}+1}=\avi(S_n).
\end{align*}
\end{proof}
It is not unexpected that the star maximises $\avi$, since it is also the tree with the greatest number of independent sets (a well-known fact first established in \cite{prodinger1982fibonacci}). The path, on the other hand, has the smallest number of independent sets among all trees of a given size. One would therefore expect that the average size of independent sets also attains its minimum for the path, which is indeed the case. Proving this fact requires more effort. Let us first find an explicit formula for the average size of independent sets of a path.
\begin{lem}\label{chap2.lem:path}
The average size of independent sets of the $n$-vertex path $P_n$ is
\begin{equation}\label{chap2.eq:tP_n}
\avi(P_n) = \frac{5 - \sqrt{5}}{10} n + \frac{3-\sqrt{5}}{5} - \frac{n+2}{\sqrt{5}((-\phi^2)^{n+2}-1)},
\end{equation}
where $\phi = \frac{\sqrt{5}+1}{2}$ is the golden ratio. In particular,
\begin{enumerate}
\item[(a)] $ \lim_{n \to \infty} \avi(P_n) - \frac{5 - \sqrt{5}}{10} n = \frac{3-\sqrt{5}}{5},$
\item[(b)] $ \avi(P_n) \geq \frac{5-\sqrt{5}}{10} n + \frac{1}{\sqrt{5}} - \frac13$, with equality only for $n = 2$. For all positive integers $n \neq 2$, we even have $ \avi(P_n) \geq \frac{5-\sqrt{5}}{10} n + \frac{2}{\sqrt{5}} - \frac34$, with equality only for $n=4$.
\end{enumerate}
\end{lem}

\begin{proof}
It is well known that the number of independent sets of $P_n$ is the Fibonacci number $F_{n+2} = \frac{1}{\sqrt{5}} \big( \phi^{n+2} - (-\phi)^{-n-2} \big)$ (see \cite{prodinger1982fibonacci}). The total number of vertices $\Ti(P_n)$ in all independent sets of $P_n$ is determined by the recursion
$$\Ti(P_n) = \Ti(P_{n-1}) + \Ti(P_{n-2}) + \I(P_{n-2})$$ that follows from Proposition \ref{sizeRec},
and the initial values $\Ti(P_1) = 1$ and $\Ti(P_2) = 2$. The solution to this recursion is easily found to be
$$\Ti(P_n) = \Big( \frac{1+\sqrt{5}}{10} n + \frac{2\sqrt{5}}{25} \Big) \phi^n + \Big( \frac{1-\sqrt{5}}{10} n - \frac{2\sqrt{5}}{25} \Big) (-\phi)^{-n}.$$

The formula~\eqref{chap2.eq:tP_n} for the quotient $\avi(P_n) = \Ti(P_n)/\I(P_n)$ follows immediately, as does the limit in (a).

Now we show that the absolute value of the error term is decreasing: for $n \geq 2$, we have
\begin{align*}
\Bigg| \frac{\frac{n+2}{\sqrt{5}((-\phi^2)^{n+2}-1)}}{\frac{n+1}{\sqrt{5}((-\phi^2)^{n+1}-1)}} \Bigg| &\leq \Big( 1 + \frac{1}{n+1} \Big) \cdot \frac{\phi^{2(n+1)}+1}{\phi^{2(n+2)}-1}\\& = \phi^{-2} \Big( 1 + \frac{1}{n+1} \Big) \cdot \frac{\phi^{-2(n+1)}+1}{1-\phi^{-2(n+2)}} \\
&\leq \phi^{-2} \cdot \frac{4}{3} \cdot \frac{\phi^{-6}+1}{1-\phi^{-8}} = \frac{4(\sqrt{5}-1)}{9} < 1.
\end{align*}
Therefore, the difference
$$\Big| \avi(P_n) - \frac{5-\sqrt{5}}{10} n - \frac{3-\sqrt{5}}{5}\Big|$$
is decreasing in $n$.
Moreover, note that the sign of $\frac{n+2}{\sqrt{5}((-\phi^2)^{n+2}-1)}$ alternates, so that $\avi(P_n)$ is alternatingly greater and less than $\frac{5  - \sqrt{5}}{10} n + \frac{3-\sqrt{5}}{5}$. It follows that the minimum of the difference $\avi(P_n) - \frac{5-\sqrt{5}}{10} n$ is attained for $n=2$. Among all $n \neq 2$, the minimum occurs when $n=4$. The values of $\avi(P_n)$ are easily calculated in both cases, and the two inequalities in (b) follow.
\end{proof}

For ease of notation, we set $a= \frac{5-\sqrt{5}}{10} \approx 0.27639320$ and $c_n = \avi(P_n) - an$. The following table gives values of $c_n$ for small $n$:

\begin{table}[htbp]
\centering
\begin{tabular}{|c|c|c|c|c|c|}
\hline
$n$ & 1 & 2 & 3 & 4& 5 \\
\hline
$c_n$ & $\frac{1}{2\sqrt{5}} \approx 0.2236$ & $\frac{1}{\sqrt{5}} - \frac13 \approx 0.1139$ & $\frac{3}{2\sqrt{5}} - \frac12 \approx 0.1708$& $\frac{2}{\sqrt{5}} - \frac34 \approx 0.1444$ & $\frac{\sqrt{5}}{2} - \frac{25}{26} \approx 0.1565$  \\
\hline
\end{tabular}
\vspace*{0.3cm}
\caption{Values of $c_1,c_2,\ldots,c_5$ for independent sets.}
\end{table}

Before we prove our main result, we require one more lemma:

\begin{lem}\label{chap2.lem:quot}
For every tree $T$ and every vertex $v$ of $T$, we have
$$\frac12\leq \frac{\I(T-v)}{\I(T)} < 1.$$
\end{lem}

\begin{proof}
Note first that $\I(T) = \I(T-v) + \I(T-N[v])$. Since $T-N[v]$ is a subgraph of $T-v$, we have $\I(T-N[v]) \leq \I(T-v)$, hence $2\I(T-v) \geq \I(T)$, which proves the first inequality. The second inequality simply follows from the fact that $T-v$ is an induced proper subgraph of~$T$.
\end{proof}

\begin{thm}\label{thm:path}
For every tree $T$ of order $n$ that is not a path, we have the inequality $\avi(T) \geq an + b$, where $b = (79\sqrt{5}-165)/70 \approx 0.16641957$. Consequently, the path minimises the value of $\avi(T)$ among all trees of order $n$.
\end{thm}

\begin{proof}
We prove the inequality by induction on $n$. For $n \leq 3$, there is nothing to prove since the only trees with three or fewer vertices are paths. Thus assume now that $n \geq 4$, and consider a vertex $v$ of the tree $T$ whose degree is at least $3$ (which must exist if $T$ is not a path). Denote the neighbours of $v$ by $v_1,v_2,\ldots,v_k$ and the components of $T - v$ by $T_1,T_2,\ldots,T_k$ (in such a way that $v_j$ is contained in $T_j$). By Proposition \ref{chap2.Eq:muRec}, we have

\begin{align*}
\avi(T) &= \frac{\Ti(T)}{\I(T)} = \frac{\Ti(T-v) + (\I(T-N[v])+\Ti(T-N[v]))}{\I(T)} \\
&= \frac{\I(T-v)}{\I(T)} \cdot \frac{\Ti(T-v)}{\I(T-v)} + \frac{\I(T-N[v])}{\I(T)} \cdot \Big(1 + \frac{\Ti(T-N[v])}{\I(T-N[v])}\Big) \\
&=\frac{\I(T-v)}{\I(T)}  \avi(T-v) + \frac{\I(T) - \I(T-v)}{\I(T)} (1+ \avi(T-N[v])) \\
&=\frac{\I(T-v)}{\I(T)}  \sum_{j=1}^k \avi(T_j) + \Big( 1 - \frac{\I(T-v)}{\I(T)} \Big) \Big(1 + \sum_{j=1}^k \avi(T_j - v_j)\Big).
\end{align*}
Assume first that $k \geq 5$, and let $T' = T - T_k$ be the tree obtained by removing $T_k$ from $T$. Repeating the calculation, we also have
$$\avi(T') = \frac{\I(T'-v)}{\I(T')}  \sum_{j=1}^{k-1} \avi(T_j) + \Big( 1 - \frac{\I(T'-v)}{\I(T')} \Big) \Big(1 + \sum_{j=1}^{k-1} \avi(T_j - v_j)\Big).$$
For simplicity, let us introduce the notations $\rho = \frac{\I(T-v)}{\I(T)}$ and $\rho' = \frac{\I(T'-v)}{\I(T')}$. Note that
\begin{equation}\label{chap2.eq:rho}
\rho = \frac{\I(T-v)}{\I(T)} = \frac{\prod_{j=1}^k \I(T_j)}{\prod_{j=1}^k \I(T_j) + \prod_{j=1}^k \I(T_j-v_j)} = \frac{1}{1+ \prod_{j=1}^k \frac{\I(T_j-v_j)}{\I(T_j)}}
\end{equation}
and likewise
$$\rho' = \frac{1}{1+ \prod_{j=1}^{k-1} \frac{\I(T_j-v_j)}{\I(T_j)}},$$
so that Lemma~\ref{chap2.lem:quot} implies $\rho > \rho'$. 
Now we write
\begin{align*}
\avi(T) &= \rho \sum_{j=1}^k \avi(T_j) + (1-\rho) \Big(1 + \sum_{j=1}^k \avi(T_j - v_j)\Big) \\
&= \rho \avi(T_k) + (1-\rho)\avi(T_k-v_k) +\rho \sum_{j=1}^{k-1} \avi(T_j) \\&\quad+ (1-\rho) \Big(1 + \sum_{j=1}^{k-1} \avi(T_j - v_j)\Big)\\
&= \rho \avi(T_k) + (1-\rho)\avi(T_k-v_k) \\
&\quad+ \frac{1-\rho}{1-\rho'} \bigg( \rho' \sum_{j=1}^{k-1} \avi(T_j) + (1-\rho') \Big(1 + \sum_{j=1}^{k-1} \avi(T_j - v_j)\Big) \bigg) \\
&\quad+ \frac{\rho-\rho'}{1-\rho'} \sum_{j=1}^{k-1} \avi(T_j).
\end{align*}
By Lemma~\ref{chap2.lem:path} and the induction hypothesis, we have $\avi(T_j) \geq a|T_j| + \frac{1}{\sqrt{5}} - \frac13$ for all $j$. It follows that
\begin{align*}
\sum_{j=1}^{k-1} \avi(T_j) &\geq \sum_{j=1}^{k-1} \Big(a|T_j|+\frac{1}{\sqrt{5}} - \frac13\Big) = a(|T'|-1) +(k-1)\Big(\frac{1}{\sqrt{5}} - \frac13\Big) \\
&\geq a|T'|+4\Big( \frac{1}{\sqrt{5}} - \frac13 \Big) - a >  a|T'|+b.
\end{align*}
Moreover, the induction hypothesis gives us $\avi(T')\geq a|T'| + b$. Finally,
\begin{itemize}
\item If $|T_k| \geq 4$, then by the induction hypothesis, Lemmas ~\ref{chap2.lem:path} and \ref{chap2.lem:quot}, we have
\begin{align*}
&\rho \avi(T_k) + (1-\rho)\avi(T_k-v_k) \\
&\geq \rho \Big(a|T_k|+ \frac{2}{\sqrt{5}} - \frac34\Big)  + (1-\rho) \Big(a(|T_k|-1)+ \frac{2}{\sqrt{5}} - \frac34\Big) \\
&=a|T_k| + \frac{2}{\sqrt{5}} - \frac34 - (1-\rho)a \geq a|T_k| + \frac{2}{\sqrt{5}} - \frac34 - \frac{a}{2} > a|T_k|.
\end{align*}
\item If $|T_k| = 3$, then $\rho \avi(T_k) + (1-\rho)\avi(T_k-v_k) \geq \rho + (1-\rho) \cdot \frac23 = \frac{2+\rho}{3} \geq \frac56 > 3a$ (by Lemma~\ref{chap2.lem:quot}).
\item If $|T_k| = 2$, then $\rho \avi(T_k) + (1-\rho)\avi(T_k-v_k) = \rho \cdot \frac23 + (1-\rho) \cdot \frac12 = \frac{3+\rho}{6} \geq \frac7{12} > 2a$ (by Lemma~\ref{chap2.lem:quot}).
\item If $|T_k| = 1$, then $\rho \avi(T_k) + (1-\rho)\avi(T_k-v_k) = \rho \cdot \frac12 + (1-\rho) \cdot 0 = \frac{\rho}{2}$, and since $\frac{\I(T_k-v_k)}{\I(T_k)} = \frac12$ in this case, we have $\rho \geq \frac23$ by~\eqref{chap2.eq:rho}. Thus $\rho \avi(T_k) + (1-\rho)\avi(T_k-v_k) \geq \frac13 > a$. 
\end{itemize}
In conclusion, $\rho \avi(T_k) + (1-\rho)\avi(T_k-v_k) > a|T_k|$. Combining all inequalities, we obtain
\begin{align*}\avi(T) &>a|T_k| + \frac{1-\rho}{1-\rho'} ( a|T'|+b) + \frac{\rho-\rho'}{1-\rho'} ( a|T'|+b) \\&= a(|T'|+|T_k|) + b = a|T|+b.
\end{align*}
This completes the case that $k \geq 5$, so we are left with the cases $k=3$ and $k=4$. We return to the representation
\begin{equation}\label{chap2.eq:mrep}
\avi(T) = \rho \sum_{j=1}^k \avi(T_j) + (1-\rho) \Big(1 + \sum_{j=1}^k \avi(T_j - v_j)\Big).
\end{equation}
Now we distinguish different cases depending on how many of the branches $T_j$ have one, two or three vertices respectively. If $T_j$ has three vertices, we also distinguish whether $v_j$ is the centre vertex or a leaf of $T_j$. This gives us a total of $35$ cases for $k=3$ and $70$ cases for $k=4$, corresponding to the solutions of
$$x_1+x_2+x_3+x_4+x_5 = k.$$
Here, $x_1$ and $x_2$ stand for the number of $T_j$'s with one and two vertices respectively, $x_3$ and $x_4$ for the number of $T_j$'s with three vertices and $v_j$ the centre ($x_3$) or a leaf ($x_4$) respectively, and $x_5$ is the number of $T_j$'s with four or more vertices. In each of the cases, we use the following explicit values and estimates:
$$\avi(T_j) \begin{cases}
= \frac12 &|T_j| = 1, \\
= \frac23 &|T_j| = 2, \\
= 1 &|T_j| = 3, \\
\geq a|T_j| + \frac{2}{\sqrt{5}} - \frac34 & \text{otherwise,}
\end{cases}$$
$$\avi(T_j-v_j) \begin{cases}
= 0 &|T_j| = 1, \\
= \frac12 &|T_j| = 2, \\
= \frac23 &|T_j| = 3 \text{ and $v_j$ is a leaf of $T_j$}, \\
= 1 &|T_j| = 3 \text{ and $v_j$ is the centre of $T_j$}, \\
\geq a(|T_j|-1) + \frac{2}{\sqrt{5}} - \frac34 & \text{otherwise,}
\end{cases}$$
$$\frac{\I(T_j-v_j)}{\I(T_j)} \begin{cases}
= \frac12 &|T_j| = 1, \\
= \frac23 &|T_j| = 2, \\
= \frac35 &|T_j| = 3 \text{ and $v_j$ is a leaf of $T_j$}, \\
= \frac45 &|T_j| = 3 \text{ and $v_j$ is the centre of $T_j$}, \\
\in [\frac12,1] & \text{otherwise.}
\end{cases}$$
The bounds on $\avi$ in the case that $T_j$ has four or more vertices follow from the induction hypothesis (if $T_j - v_j$ is disconnected, applied to all components), while the last inequality is simply Lemma~\ref{chap2.lem:quot}.

We plug these bounds into~\eqref{chap2.eq:mrep} and also use the identity~\eqref{chap2.eq:rho} again. Since the expression~\eqref{chap2.eq:mrep} is linear in $\rho$, its minimum is either attained for the largest or smallest possible value of $\rho$. This gives us a lower bound for $\avi(T)$ in each of the aforementioned $105$ cases, which can all be checked easily with a computer.  As an example, let us consider the case that gives us the worst bound: it is obtained for $x_1=x_3=x_4=0$, $x_2=1$ and $x_5=2$ (i.e., one branch with two vertices, two branches with four or more vertices). Let $T_1$ and $T_2$ both have four or more vertices, so that the third branch $T_3$ consists of only two vertices. We have
\begin{align*}
\avi(T_1) &\geq a|T_1|+ \frac{2}{\sqrt{5}} - \frac34,\\
\avi(T_2) &\geq a|T_2|+ \frac{2}{\sqrt{5}} - \frac34,\\
\avi(T_3) &= \frac23
\end{align*}
and
\begin{align*}\avi(T_1-v_1) &\geq a|T_1| - a + \frac{2}{\sqrt{5}} - \frac34,\\
 \avi(T_2-v_2) &\geq a|T_2| - a + \frac{2}{\sqrt{5}} - \frac34,\\ \avi(T_3-v_3) &= \frac12
\end{align*}
as well as
$$\rho = \frac{1}{1+ \frac23 \frac{\I(T_1-v_1)\I(T_2-v_2)}{\I(T_1)\I(T_2)}} \in \Big[ \frac{3}{5}, \frac{6}{7} \Big]$$
by Lemma \ref{chap2.lem:quot}. Thus
\begin{align*}
\avi(T_1)+\avi(T_2)+\avi(T_3) &\geq a(|T_1|+|T_2|) + 2 \Big( \frac{2}{\sqrt{5}} - \frac34 \Big) + \frac{2}{3} \\
&= a(|T|-3) + \frac{4}{\sqrt{5}} - \frac56 \\
&= a|T| + \frac{11}{2\sqrt{5}} - \frac{7}{3}
\end{align*}
and likewise
\begin{align*}
\avi(T_1-v_1)+\avi(T_2-v_2)+\avi(T_3-v_3) &\geq a(|T_1|-1+|T_2|-1) + 2 \Big( \frac{2}{\sqrt{5}} - \frac34  \Big) + \frac12 \\
&= a(|T|-5) + \frac{4}{\sqrt{5}} - 1  \\
&= a|T| +\frac{13}{2\sqrt{5}} - \frac{7}{2}.
\end{align*}
Plugging all these inequalities into~\eqref{chap2.eq:mrep}, we obtain
\begin{align*}
\avi(T) &\geq \rho \Big( a|T| + \frac{11}{2\sqrt{5}} - \frac{7}{3} \Big) + (1-\rho) \Big(1 + a|T| + \frac{13}{2\sqrt{5}} - \frac{7}{2} \Big) \\
&= a|T| + \frac{13}{2\sqrt{5}} - \frac{5}{2} + \rho \Big( \frac16 - \frac{1}{\sqrt{5}} \Big) \geq a|T| + \frac{13}{2\sqrt{5}} - \frac{5}{2} + \frac67
\Big( \frac16 - \frac{1}{\sqrt{5}} \Big) \\
&= a|T| + b.
\end{align*}
The other cases are treated in the same fashion and give lower bounds with larger constant terms. To complete the proof of the theorem, it only remains to prove an upper bound on $\avi(P_n)$. However, we already know from Lemma~\ref{chap2.lem:path} that
\begin{align*}
\avi(P_n)  &= an +\frac{3-\sqrt{5}}{5} - \frac{n+2}{\sqrt{5}((-\phi^2)^{n+2}-1)}\\& \leq an +\frac{3-\sqrt{5}}{5} - \frac{7}{\sqrt{5}((-\phi^2)^{7}-1)} = an + \frac{\sqrt{5}}{2} - \frac{25}{26}
\end{align*}
for $n > 3$, and $\frac{\sqrt{5}}{2} - \frac{25}{26} \approx 0.15649553 < b$. Therefore, $\avi(P_n) < an + b \leq \avi(T)$ for every tree $T$ with $n$ vertices other than $P_n$. This completes the proof.
\end{proof}

\section{A more general setting}

It is common in statistical physics to consider the hard-core distribution on the independent sets $I$ of a graph $G$. That is, the study of a random independent set $I$ with probability proportional to $\A^{|I|}$. In \cite{davies2018average, davies2017independent}, the authors consider this model and prove bounds on the expected size of an independent set drawn from the hard-core model on $G$ at fugacity $\A$. When $\alpha = 1$, this expected size is precisely the invariant $\avi$ that we investigated in this paper.
Recall that $\ii(G,k)$ is the number of independent vertex subsets of size $k$ in $G$. Now, choose a random independent set with probability proportional to $\alpha^k$, where $k$ is the size of the set and $\A$ is a positive number. We define the weighted total number of independent subsets of $G$, the weighted total size of independent subsets of $G$ and the weighted average size of independent vertex subsets in $G$:

\begin{align*}
I^{\A}(G) &= \I(G,\A)=\sum_{k\geq 0}\ii(G,k)\A^k,\\
T^{\A}(G) &= \Ti(G,\A)=\sum_{k\geq 0}k\ii(G,k)\A^k,\\
\avi^{\A}(G) &=\frac{\Ti(G,\A)}{\I(G,\A)}.
\end{align*}

\begin{exa}
For the $n$-vertex edgeless graph $E_n$ and the star $S_n$, we have
\begin{align*}
\I^{\A}(E_n)&=\sum_{k= 0}^n\binom{n}{k}\A^k=(1+\A)^{n},\quad \I^{\A}(S_n)=\A+\sum_{k= 0}^{n-1}\binom{n-1}{k}\A^k=\A+(1+\A)^{n-1},\\
\Ti^{\A}(E_n)&=\sum_{k=0}^n\binom{n}{k}k\A^k=\A n(1+\A)^{n-1},\quad \Ti^{\A}(S_n)=\A+\A(n-1)(1+\A)^{n-2},\\
\avi^{\A}(E_n)&=\frac{\A n}{1+\A}, \quad \avi^{\A}(S_n)=\frac{\A+\A(n-1)(1+\A)^{n-2}}{\A+(1+\A)^{n-1}}.
\end{align*}
\end{exa}

All the results presented in this paper, except for the extremality of the path, generalise to this weighted average. The proof that the path is extremal generalises to the case that $\A \in (0,1]$, but not to all real values of $\A$ (in fact, the path is not extremal for large values of $\A$). 
To some extent, this also explains why proving extremality of the path is harder than proving extremality of the star.
We refer to \cite{Valisoa} for more details.


\begin{thebibliography}{10}

\bibitem{brightwell2002hard}
G.~R. Brightwell and P.~Winkler.
\newblock Hard constraints and the {B}ethe lattice: adventures at the interface
  of combinatorics and statistical physics.
\newblock In {\em Proceedings of the {I}nternational {C}ongress of
  {M}athematicians, {V}ol. {III} ({B}eijing, 2002)}, pages 605--624. Higher Ed.
  Press, Beijing, 2002.

\bibitem{davies2017independent}
E.~Davies, M.~Jenssen, W.~Perkins, and B.~Roberts.
\newblock Independent sets, matchings, and occupancy fractions.
\newblock {\em J. Lond. Math. Soc. (2)}, 96(1):47--66, 2017.

\bibitem{davies2018average}
E.~Davies, M.~Jenssen, W.~Perkins, and B.~Roberts.
\newblock On the average size of independent sets in triangle-free graphs.
\newblock {\em Proc. Amer. Math. Soc.}, 146(1):111--124, 2018.

\bibitem{haslegrave2014extremal}
J.~Haslegrave.
\newblock Extremal results on average subtree density of series-reduced trees.
\newblock {\em J. Combin. Theory Ser. B}, 107:26--41, 2014.

\bibitem{jamison1983average}
R.~E. Jamison.
\newblock On the average number of nodes in a subtree of a tree.
\newblock {\em J. Combin. Theory Ser. B}, 35(3):207--223, 1983.

\bibitem{jamison1984monotonicity}
R.~E. Jamison.
\newblock Monotonicity of the mean order of subtrees.
\newblock {\em J. Combin. Theory Ser. B}, 37(1):70--78, 1984.

\bibitem{levit2005independence}
V.~E. Levit and E.~Mandrescu.
\newblock The independence polynomial of a graph---a survey.
\newblock In {\em Proceedings of the 1st {I}nternational {C}onference on
  {A}lgebraic {I}nformatics}, pages 233--254. Aristotle Univ. Thessaloniki,
  Thessaloniki, 2005.

\bibitem{merrifield1989topological}
R.~E. Merrifield and H.~E. Simmons.
\newblock {\em Topological {M}ethods in {C}hemistry}.
\newblock Wiley, New York, 1989.

\bibitem{Mol}
L.~Mol and O.~Oellermann.
\newblock Maximizing the mean subtree order.
\newblock arXiv:1707.01874, 2017.

\bibitem{prodinger1982fibonacci}
H.~Prodinger and R.~F. Tichy.
\newblock Fibonacci numbers of graphs.
\newblock {\em Fibonacci Quart.}, 20(1):16--21, 1982.

\bibitem{Valisoa}
V.~Razanajatovo~Misanantenaina.
\newblock {\em Properties of graph polynomials and related parameters}.
\newblock PhD thesis, Stellenbosch University, 2017.

\bibitem{stephens2018mean}
A.~M. Stephens and O.~R. Oellermann.
\newblock The mean order of sub-{$k$}-trees of {$k$}-trees.
\newblock {\em J. Graph Theory}, 88(1):61--79, 2018.

\bibitem{vince2010average}
A.~Vince and H.~Wang.
\newblock The average order of a subtree of a tree.
\newblock {\em J. Combin. Theory Ser. B}, 100(2):161--170, 2010.

\bibitem{wagner2017upper}
S.~Wagner.
\newblock Upper and lower bounds for {M}errifield-{S}immons index and {H}osoya
  index.
\newblock In I.~Gutman, B.~Furtula, K.~C. Das, E.~Milovanovi{\'{c}}, and
  I.~Milovanovi{\'{c}}, editors, {\em Bounds in Chemical Graph Theory --
  Basics}, volume~19 of {\em Mathematical Chemistry Monographs}, pages
  155--187. University of Kragujevac and Faculty of Science Kragujevac, 2017.

\bibitem{wagner2016local}
S.~Wagner and H.~Wang.
\newblock On the local and global means of subtree orders.
\newblock {\em J. Graph Theory}, 81(2):154--166, 2016.

\end{thebibliography}
\end{document}